\documentclass[12pt]{article}

\usepackage{inputenc}
\usepackage{diagbox}
\usepackage{mathtools}
\usepackage{bbm}
\usepackage{latexsym}
\usepackage{epsfig}
\usepackage{amsmath,amsthm,amssymb,enumerate}

\usepackage[T1]{fontenc}

\usepackage[a-1b]{pdfx}
\usepackage{hyperref}

\usepackage{color}

\def\cH{{\mathcal H}}

\def\nn{\nonumber}
\def\a{\alpha}   \def\D{\Delta}
\def\e{\varepsilon} \def\f{\phi}   
  \def\k{\kappa}
     
  \def\n{\nu} \def\p{\pi}
   
 \def\om{\omega}

\newcommand{\upp}[1]{\langle#1\rangle}

\newtheorem{theorem}{Theorem}
\newtheorem{lemma}[theorem]{Lemma}
\newtheorem{corollary}[theorem]{Corollary}

\newtheorem{Proposition}{Proposition}

\newcommand{\rdup}[1]{{\left\lceil #1\right\rceil }}

\newcommand{\brac}[1]{\left(#1\right)}

\newcommand{\bfrac}[2]{\left(\frac{#1}{#2}\right)}

\newcommand{\set}[1]{\left\{#1\right\}}

\def\E{\mathbb{E}}
\def\Var{\mathbb{V}ar}
\def\Pr{\mathbb{P}}

\newcommand{\ignore}[1]{}

\newcommand{\THM}[2]{\
\begin{theorem}\label{#1}#2
\end{theorem}
}

\def\cH{{\mathcal H}}

\def\cR{{\mathcal R}}
\def\cS{{\mathcal S}}

\newcommand{\beq}[2]{\begin{equation}\label{#1}#2\end{equation}}
\newcommand{\mults}[1]{\begin{multline*}#1\end{multline*}}

\def\nn{\nonumber}

\usepackage{tikz}
\usetikzlibrary{decorations.pathmorphing}
\usetikzlibrary{positioning}
\usetikzlibrary{arrows,automata}
\usetikzlibrary{shapes.misc}
\usetikzlibrary{backgrounds}
\usetikzlibrary{arrows,shapes}

\begin{document}
\author{Tolson Bell\thanks{Email: thbell@cmu.edu. Research supported by NSF Graduate Research Fellowship grant DGE2140739 and DMS1952285.}~~and Alan Frieze\thanks{Email: frieze@cmu.edu. Research supported in part by NSF grant DMS1952285.}\\Department of Mathematical Sciences\\Carnegie Mellon University\\Pittsburgh PA 15213}
\title{Rainbow powers of a Hamilton cycle in $G_{n,p}$}
\maketitle
\begin{abstract}
We show that the threshold for having a rainbow copy of a power of a Hamilton cycle in a randomly edge colored copy of $G_{n,p}$ is within a constant factor of the uncolored threshold. Our proof requires $(1+\e)$ times the minimum number of colors.
\end{abstract}
\section{Introduction}
There has recently been great progress in our understanding of thresholds for monotone properties in the random graph $G_{n,p}$. Inspired by the work of Alweiss, Lovett, Wu and Zhang \cite{ALWZ} on the Sunflower Conjecture, Frankston, Kahn, Narayanan and Park \cite{FKNP2020} showed that under fairly general conditions, the threshold for the existence of combinatorial objects is within a factor $O(\log n)$ of the point where the expected number of such objects begins to take off. Great though these results are, this is not the end of the story. In a paper remarkable for the strength of its result and for the simplicity of its proof, Park and Pham\cite{PP22} proved the so-called Kahn-Kalai conjecture \cite{KK} which implies the result of \cite{FKNP2020}.

Kahn, Narayanan and Park \cite{KNP2020} tightened their analysis for the case of the square of a Hamilton cycle, removing the $O(\log n)$ factor and solving the existence problem up to a constant factor; a remarkable achievement, given the complexity of the proofs of earlier weaker results. Their result was generalized by Espuny D\'iaz and Person \cite{Esp} and Spiro \cite{Spiro}, both of whom defined more generalized conditions under which the $O(\log n)$ factor can be removed. Espuny D\'iaz and Person asked whether a rainbow generalization of their result could be proven \cite{Esp}. Our main theorem here proves a rainbow version in a setting that is more general than the Kahn--Narayanan--Park result but less general than Espuny D\'iaz--Person and Spiro results. It is likely that our result could be extended to the full generality of the Espuny D\'iaz--Person and Spiro results with some additional effort.

\paragraph{Some notation} Given a set $X$ and $0\leq p\leq 1$, we let $X_p$ denote a subset of $X$ where each $x\in X$ is placed independently into $X_p$ with probability $p$. Similarly, $X_m$ is a random $m$-subset of $X$ for $1\leq m\leq |X|$.

Let $\cH=\set{A_1,A_2,\ldots,A_M}$ be a hypergraph on vertex set $X$.
A key notion in this analysis is that of {\em spread}. For a set $S\subseteq X$ we let $\upp{S}=\set{T:\;S\subseteq T\subseteq X}$ denote the subsets of $X$ that contain $S$.  We say that $\cH$ is $\k$-spread if
\beq{spread}{
|\cH\cap \upp{S}|\leq \frac{|\cH|}{\k^{|S|}},\quad\forall S\subseteq X.
}
$\cH$ is called $r$-bounded if $|A|\leq r$ for all $A\in\cH$ and $r$-uniform if $|A|=r$ for all $A\in\cH$. The following theorem was proved in \cite{FKNP2020}:
\THM{T2020}{
Let $\cH$ be an $r$-bounded, $\k$-spread hypergraph and let $X=V(\cH)$. There is an absolute constant $K>0$ such that if 
\beq{mbound}{
p\geq\frac{K\log r}{\k}\text{ or respectively }m\geq\frac{(K\log r)|X|}{\k}
}
then w.h.p. $X_p$ or $X_m$ respectively contains an edge of $\cH$. More precisely, $\Pr(X_p\text{ contains an edge of $\cH$ })\geq 1-\e_r$ where $\e_r\to0$ as $r\to\infty$.
}
To apply this theorem to, say, Hamilton cycles, we let $X=\binom{[n]}{2}$ and we let $A_i,i=1,2,\ldots,\tfrac12(n-1)!$ be the edge sets of the Hamilton cycles of $K_n$.

In the special case of $\cH$ corresponding in this way to the squares of Hamilton cycles, \cite{KNP2020} removed the $\log r$-factor from the bounds in \eqref{mbound}.

We now turn to the main topic of this note. We suppose that each $x\in X$ is {uniformly and independently} given a random color from a set $Q$. Given a set $A\subseteq X$ we refer to $A^*$ as the set after its elements have been colored. We say that $A^*$ is {\em rainbow} colored if each $a\in A$ has a different color. Bell, Frieze and Marbach \cite{BFM} attempted to extend the results of \cite{FKNP2020} to rainbow colorings. They proved
\THM{thrainbow}{
Let $\cH$ be an $r$-bounded, $\k$-spread hypergraph and let $X=V(\cH)$ be randomly colored from $Q=[q]$ where $q\ge r$. Suppose also that $\k=\Omega(r)$, that is, there exists a constant $C_0>0$ such that $\kappa\ge C_0r$ for all valid $r$. Then given $\e>0$ there is a constant $C_\e$ such that if $r$ is sufficiently large and
\beq{mbounds}{
m\geq\frac{(C_\e\log_2r)|X|}{\k}
}
then $X_m$ contains a rainbow colored edge of $\cH$ with probability at least $1-\e$.  
}
The constraint $\k=\Omega(r)$ rules out the square of a Hamilton cycle as there we have $r=2n$ and $\k=O(n^{1/2})$. The aim of this note is to tackle this case while also removing the extra $\log r$-factor. {Unfortunately, we have to increase the number of colors slightly, by a factor $(1+\e_1)$ for arbitrary positive $\e_1$. Chapter 15 of \cite{FK} extracts a property used in  \cite{KNP2020}} to make the following extra assumption about the hypergraph $\cH$. For $A\in\cH$ we let
\[
f_{t,A}=|\set{B\in \cH:|B\cap A|=t}|.
\]
The assumption now is that there exist constants $0<\a<1,K_0$ independent of $r$ such that
\beq{p1}{
f_{t,A}\leq \bfrac{K_0}{\k}^t|\cH|\text{ for all $A\in\cH$ and }1\leq t\leq \a r.
}
As $\cH$ remains $\kappa$-spread, it follows from \eqref{spread} that
\beq{spreadf}{
f_{t,A}\le\frac{2^r}{\k^t}|\cH|\text{ for all $A\in\cH$ and }t>\a r
}
Let a hypergraph $\cH$ be {\em edge transitive} if for every pair of edges $A_i,A_j$ there exists a permutation $\p:X\to X$ such that $\p(A_i)=A_j$ and such that $\p(A)\in \cH$ for all $A\in \cH$. The induced map $\p:\cH\to\cH$ is a bijection. (When $\cH$ is defined by the edges of $K_n$, all we usually require is a permutation of the vertices.)

We will prove the following:
\THM{power}{
Let $\e,\e_1>0$ be arbitrary positive constants. Suppose that $\cH$ is a $\k$-spread, $r$-uniform and edge transitive hypergraph on which \eqref{p1} holds. Let $X=V(\cH)$ be randomly colored from $Q=[q]$ where $q\geq (1+\e_1)r$. Then there exists $C=C(\e,\e_1)$ {such that for sufficiently large $r,\k$,}
\beq{strictly}{
m\geq \frac{C|X|}{\k} \text{ implies that }{\Pr(X_{m}^*\text{ contains a rainbow colored edge of $\cH$}})\geq 1-\e.
}
}
We will show in Section \ref{appen} that hypergraphs corresponding to powers of Hamilton cycles fit the premise of Theorem \ref{power}. (\cite{KNP2020} verified \eqref{p1} for squares of Hamilton cycles and for completeness, we verify \eqref{p1} for all powers.) We prove Theorem \ref{power} in the next section. We note that our proof is in some part inspired by a proof by Huy Pham \cite{HP} of the main result of \cite{KNP2020}.
\section{Proof of Theorem \ref{power}}
{The proof will proceed in three stages. First, we will color all elements of $X$ independently and uniformly at random from $[q]$, and will remove all sets in $\cH$ that are not rainbow. We show that the number of remaining sets is with high probability close to its expectation.}

Then, let $N=|X|$ and $m=\frac{CN}{\k}$ for sufficiently large $C=C(\e,\e_1)$.  Let $W_0$ be chosen randomly from $\binom{X}{m}$. Let $p_1=\frac{m}{N}$ and let $W_1$ be obtained from $X\setminus W_0$ by including each element with probability $p_1$. Proving Theorem \ref{power} on $W_0\cup W_1$ suffices to prove it for $X_{O(m)}$ by standard concentration bounds. {The second stage (succeeding with high probability) will deal with $W_0$ while the third stage (succeeding with probability $1- \e$) will deal with $W_1$.}

We will use the notation $A\lesssim B$ to indicate that $A\leq (1+o(1))B$ as $r\to \infty$. We will also assume that $q=(1+\e_1)r$. This assumption comes without loss of generality because $C(\e,\e_1)$ will be strictly decreasing in $\e_1$, so if $q>(1+\e_1)r$, we could set $\e_2$ such that $q=(1+\e_2)r$ and use $\e_2$ in the proof instead.
\subsection{The size of $\cH^*$}
Let $\cH=\set{A_1,A_2,\ldots,A_M}$ and let $\cH^*$ denote the rainbow edges of $\cH$ after a uniform and independent random coloring. Similarly, let $X^*$ denote $X$ after it has been randomly colored. Let $(a)_b=a(a-1)\cdots(a-b+1)$ for positive integers $a,b$.

We use the Chebyshev inequality to prove concentration of $Z=|\cH^*|$ around its mean. We have
\[
\E(Z)=\frac{|\cH|(q)_r}{q^r}\to\infty.
\]
{as $r\to\infty$, because spread (with $S\in\cH$ in \eqref{spread}) implies that $|\cH|\geq \k^r$ and we have assumed that $\k$ is sufficiently large.}

Using the edge transitivity of $\cH$ to obtain \eqref{esy},
\begin{align}
\E(Z^2)&=\sum_{t=0}^r\sum_{A_i,A_j:|A_i\cap A_j|=t} \frac{(q)_t((q-t)_{r-t})^2}{q^{2r-t}}\nn\\
&\leq\E(Z)\brac{1+\E(Z)+\sum_{t=1}^{r-1}\sum_{A_i:|A_i\cap A_1|=t}\frac{(q-t)_{r-t}}{q^{r-t}}}\label{esy}\\
&\leq \E(Z)\brac{1+\E(Z)+|\cH|\brac{\sum_{t=1}^{\a r}\bfrac{K_0}{\k}^t\frac{(q-t)_{r-t}}{q^{r-t}}+\sum_{t=\a r+1}^{r-1}\frac{\binom{r}{t}}{\k^t}\frac{(q-t)_{r-t}}{q^{r-t}}}}.\label{fsum}
\end{align}
{\bf Explanation for \eqref{fsum}:} For the first sum we use \eqref{p1} on $A_i$ and for the second sum we use spread by summing over all $\binom{r}{t}$ $t$-subsets of $A_i$.

So,
\[
\frac{\E(Z^2)}{\E(Z)^2}\leq \frac{1}{\E(Z)}+1+\sum_{t=1}^{\a r}\bfrac{K_0}{\k}^t\frac{(q-t)_{r-t}q^r}{q^{r-t}(q)_r}+\sum_{t=\a r+1}^{r-1}\frac{2^r}{\k^t}\frac{(q-t)_{r-t}q^r}{q^{r-t}(q)_r}=1+o(1)
\]
as long as $\kappa,r\to\infty$. It follows that w.h.p.
\beq{cA}{
|\cH^*|\sim \frac{|\cH|(q)_r}{q^r}.
}
Thus, for the rest of the proof we will assume $|\cH^*|\ge (1-o(1))\frac{|\cH|(q)_r}{q^r}$.
\subsection{Random sample from $X$}
Given a set $A^*\in\cH^*$, we define
\[
f^*_{t,A^*}=|\set{B^*\in \cH^*:|B^*\cap A^*|=t}|
\]
so that for $1\le t\le\alpha r$, we have by \eqref{p1} that
\beq{expf}{
\E(f^*_{t,A^*})= \frac{(q-t)_{r-t}}{q^{r-t}}f_{t,A}\leq \frac{(q-t)_{r-t}}{q^{r-t}}\bfrac{K_0}{\k}^t|\cH|.
}
and for $t>\alpha r$, we have by \eqref{spreadf} that
\beq{expfbig}{
\E(f_{t,A^*}^*)=\frac{(q-t)_{r-t}}{q^{r-t}}f_{t,A}\le\frac{(q-t)_{r-t}}{q^{r-t}}\frac{2^r}{\k^t}|\cH|
}

For $A^*\in\cH^*$ and $W_0^*\subseteq X^*$ with $|W_0^*|=m$, let $T^*=T^*(A^*,W_0^*)$ be $B^*\setminus W_0^*$ for some $B^*\in\cH^*,B^*\subseteq A^*\cup W_0^*$ that minimizes $|B^*\setminus W_0^*|$.

Let $\om\to\infty,\om=o(r^{1/2})$. For $A^*\in\cH^*$ we say that $(A^*,W_0^*)$ is {\em bad} if $|T^*(A^*,W_0^*)|\ge\om$. Otherwise $(A^*,W_0^*)$ is {\em good}. Let $W_0^*$ be a {\em success} if $|\set{A^*\in\cH^*:(A^*,W_0^*)\text{ is bad}}|\leq |\cH^*|/2$, that is, if the majority of sets in $\cH^*$ have a relatively small $T^*$.
\begin{lemma}\label{clt2020a}
$\Pr(\text{success})\geq 1-c_0^\om$ for some constant $0<c_0<1$.
\end{lemma}
\begin{proof}
{Let $\n_{bad}$ denote the number of bad pairs $(A^*,W_0^*)$. Fix a function $\f:2^{X^*}\rightarrow\cH^*$, where $\f(S^*)\subseteq S^*$ whenever $S^*$ contains a set in $\cH^*$. We claim that}
\begin{align}
\n_{bad}\leq \sum_{t\geq \om}\sum_{|Z^*|=m+t}\sum_{t'\geq t}2^{t'} f^*_{t',\f(Z^*)}.\label{badpairs1}
\end{align}
{\bf Explanation for \eqref{badpairs1}:} {This equation follows from the key observation of recent threshold papers \cite{KNP2020,PP22}. We count the number of $(A^*,W_0^*)$ with $|T^*(A^*,W_0^*)|=t$ for a given $t\ge\om$. We first fix $Z^*=T^*\cup W_0^*$, which as these are disjoint has size $m+t$. Then, we let $t'=|\f(Z^*)\cap A^*|$, noting that $\f(Z^*)\subseteq Z^*$ as $Z^*$ does contain a set in $\cH^*$. Since $T^*\subseteq Z^*$ is chosen to minimize $|B^*\setminus W_0^*|,B^*\in \cH^*,B^*\subseteq A^*\cup W_0^*$, and $\f(Z^*)$ is a valid choice of $B^*$, we must have $T^*\subseteq \f(Z^*)\cap A^*$, and so $t'=|\f(Z^*)\cap A^*|\geq t$. Given $t'$, we can specify one of the at most $f^*_{t',\f(Z^*)}$ possibilities for $A^*$ as a superset of $\f(Z^*)\cap A^*$. We then specify $T^*\subseteq \f(Z^*)\cap A^*$ in at most $2^{t'}$ ways, which uniquely gives $W_0^*=Z^*\setminus T^*$.}

By linearity of expectation and Equations \eqref{expf}, \eqref{expfbig}, and \eqref{badpairs1}, we get
\begin{align}
\E(\n_{bad})&\leq {\sum_{t\geq \om}}\sum_{|Z^*|=m+t} \brac{\sum_{t'=t}^{\a r}\frac{(q-t')_{r-t'}}{q^{r-t'}} \bfrac{2K_0}{\k}^{t'}|\cH|+ \sum_{t'>\a r}\frac{(q-t')_{r-t'}} {q^{r-t'}}\frac{2^{r+t'}}{\k^{t'}}|\cH|}\\
&\leq  (1+o(1))\sum_{t\geq \om}\sum_{|Z^*|=m+t} \brac{\sum_{t'=t}^{\a r}\frac{(q-t')_{r-t'}} {q^{r-t'}}\bfrac{2K_0}{\k}^{t'}\frac{q^r|\cH^*|}{(q)_r}+ \sum_{t'>\a r} \frac{(q-t')_{r-t'}}{q^{r-t'}}\frac{2^{r+t'}}{\k^{t'}}\frac{q^r|\cH^*|}{(q)_r}}\nn\\
&\leq (1+o(1))\sum_{t\geq \om}\sum_{|Z^*|=m+t} \brac{\sum_{t'=t}^{\a r} \bfrac{2eK_0}{\k}^{t'}|\cH^*|+ \sum_{t'>\a r} \frac{e^{t'}2^{r+t'}}{\k^{t'}}|\cH^*|}. \label{badpairs4}
\end{align}

Continuing, and using \eqref{badpairs4}, 
\begin{align*}
(1-o(1))~\E(\n_{bad})&\leq\sum_{t\geq \om}\binom{N}{m+t}  \brac{\sum_{t'=t}^{\a r}\bfrac{2eK_0}{\k}^{t'}|\cH^*|+\sum_{t'>\a r} \frac{2^{r+t'}e^{t'}}{\k^{t'}}|\cH^*|}\\
&\leq\binom{N}{m}|\cH^*| \sum_{t\geq \om}\bfrac{\k}{C}^t \brac{\sum_{t'=t}^{\a r}\bfrac{2eK_0}{\k}^{t'}+ \sum_{t'>\a r}\frac{2^{r+t'}{e^{t'}}}{\k^{t'}}}\\
&\leq\binom{N}{m}|\cH^*|c^\om\text{ for some }0<c<1.
\end{align*}
{Now, let $w_{bad}=|\set{W_0^*:\text{there are at least $|\cH^*|/2$ bad $(A^*,W_0^*)$}}|$. Then the above equation gives that
\[
(1-o(1))~\E(w_{bad})\leq 2\binom{N}{m}c^\om
\]
and thus
\[
\Pr(failure)=\frac{\E(w_{bad})}{ \binom{N}{m}}\leq (1+o(1))~2c^\om.
\]}
\end{proof}
{By taking $\om\to\infty$ as $r\to\infty$, this means that success will happen with high probability.}
\subsection{Finishing the proof}\label{fin}
Suppose now that $W_0^*$ is a success and then let $\cR^*$ denote the multi-{hypergraph} 
\[
\set{{T^*(A^*,W_0^*)} :A^*\in\cH^*,(A^*,W_0^*)\text{ is good}}
\]
 where each good $(A^*,W_0^*)$ contributes one element. Then let 
\[
\cR_\ell^*=\set{R^*\in \cR^*:|R^*|=\ell}\text{ for }0\leq \ell< \om.
\] 
We can assume that $\cR_0^*=\emptyset$, as otherwise $W_0^*$ contains an edge of $\cH^*$ and we have already succeeded. Now, generate $W^*=W_0^*\cup W_1^*$ where $W_1^*$ is distributed as $(X^*\setminus W_0^*)_{p_1}$. If $R^*\subseteq W_1^*$ for some $R^*\in \cR^*$, then the $B^*\in\cH^*$ for which $R^*=B^*\setminus W_0^*$ satisfies Theorem \ref{power}. Thus, we just need to show that with probability at least $1- \e$ there exists such an $R^*\subseteq W_1^*$.
 
To aid in the calculations below, for each $R^*\in\cR_{\ell}^*$ with $R^*\subseteq W_1^*$, take $A(R^*)$ to be an independent random variable with distribution Bernoulli$((\e_1p_1)^{\omega-\ell})$. $R^*\in\cR^*$ is {\em accepted} if $R^*\subseteq W_1^*$ and $A(R^*)=1$. Let $\n_R$ denote the number of accepted sets. It suffices to show $\Pr(\n_R=0)\le \e$, which we will do by Chebyshev's inequality. Then
\beq{imom}{
\E(\n_R)=\sum_{\ell=1}^\om|\cR^*_\ell|\frac{p_1^\ell(q-r+\ell)_\ell}{q^\ell}(\e_1p_1)^{\om-\ell}\sim |\cR^*|(\e_1p_1)^\om\to\infty.
}
{The claims in \eqref{imom} follow from the fact that
\beq{assum}{{
 \text{$\om=o(r^{1/2})$ and the fact that $|\cR^*|\geq \tfrac12|\cH^*|\gtrsim \tfrac12 e^{-r}|\cH|\geq \tfrac12(\k/e)^r$.}
}
}}
Now
\begin{align}
\Var(\n_R)&\leq \sum_{t=1}^\om\sum_{\ell_1,\ell_2=1}^\om\E\brac{|\set{(R^*,S^*):R^*\in\cR_{\ell_1}^*,S^*\in\cS_{\ell_2}^*,|R^*\cap S^*|=t}|}\times\nn\\
&\hspace{1in}\frac{p_1^{\ell_1}(q-r+\ell_1)_{\ell_1}}{q^{\ell_1}}\frac{p_1^{\ell_2-t}(q-r+\ell_2-t)_{\ell_2-t}}{q^{\ell_2-t}}\cdot(\e_1p_1)^{2\om-\ell_1-\ell_2}\nn\\
&\sim \sum_{t=1}^\om\E\brac{|\set{R^*,S^*\in \cR^*:|R^*\cap S^*|=t}|} (\e_1p_1)^{2\om-t}.\label{2mom}
\end{align}
{(The same assumptions \eqref{assum} suffice to obtain \eqref{2mom}.)

Fix $R^*\in \cR^*$  and then for $1\leq t\leq \om$,
\begin{align}
\E(|\set{S^*\in\cR^*:|R^*\cap S^*|=t}|)&\leq \sum_{s=t}^r\bfrac{K_0}{\k}^s\frac{(q)_{r-s}}{q^{r-s}}|\cH|\label{n1}\\
&\leq \bfrac{K_0}{\k}^t\frac{(q)_{r-t}}{q^{r-t}}|\cH|\sum_{s=t}^r\bfrac{(1+\e_1)K_0}{\e_1\k}^{s-t}\nn\\
&\leq  2 \bfrac{K_0}{\k}^t\frac{(q)_{r-t}}{q^{r-t}}|\cH|\lesssim 2\bfrac{K_0}{\k}^t\frac{(q)_{r-t}}{q^{r-t}}\frac{q^r}{(q)_r}|\cH^*|\leq 2\bfrac{eK_0}{\k}^t|\cH^*|.\nn
\end{align}
{\bf Explanation for \eqref{n1}:}
$R^*$ appears several times in $\cR^*$ as $A^*\setminus W_0^*$ for some $A^*\in\cH^*$. For each such $A^*$ we count the number of sets $B^*\in\cH^*$ for which $s=|B^*\cap A^*|\geq t$. This will bound the number of choices for $S^*$ in the LHS of \eqref{n1}. For the sum we use \eqref{expfbig} which is only valid for $t\leq \a r$. For larger $t$, we proceed as in \eqref{fsum} and $K_0^t$ by $\binom{r}{t}\leq (e/\a)^t$ and assume that $K_0\geq e/\a$.
}

So 
\mults{
\Var(\n_R)\lesssim 2|\cH^*||\cR^*|\sum_{t=1}^\om\bfrac{eK_0}{\k}^t(\e_1p_1)^{2\om-t}\\
\leq4|\cR^*|^2(\e_1p_1)^{2\om}\sum_{t=1}^\om \bfrac{eK_0}{\e_1\k p_1}^t\leq4|\cR^*|^2 (\e_1p_1)^{2\om}\sum_{t=1}^\om \bfrac{eK_0}{\e_1C}^{t}\leq \frac{12K_0}{\e_1C}\E(\n_R)^2.
}
(We have used $\k p_1=\k m/N=C$ and {$C\gg K_0$} to get the third inequality.)

The Chebyshev inequality implies that
\[
\Pr(\n_R=0)\leq \frac{\Var(\n_R)}{\E(\n_R)^2}\lesssim \frac{12K_0}{\e_1C}.
\]
Taking $C(\e,\e_1)\ge\frac{13K_0}{\e\e_1}$ then verifies \eqref{strictly}. (We use $\E(\n_R)\to\infty$ to justify the final conclusion.)
\section{Powers of Hamilton cycles}\label{appen}
We verify \eqref{p1} for the hypergraph $\cH$ whose edges correspond to the $k$th power of a Hamilton cycle. As in \cite{KNP2020} we split this into two propositions and modify their proof for $k=2$.
\begin{Proposition}\label{prop1}
For $T\subseteq \binom{[n]}{2}$, with $t\leq n/3k$ edges, inducing $c$ components,
\[
|\cH\cap\upp{T}|\leq (2k)^{2t}\brac{n-\rdup{\frac{t+(2k-1)c}{k}}+c-1}!.
\]
\end{Proposition}
\begin{proof}
Let $T_1,\ldots,T_c$ be the components of the subgraph induced by the edges $T$ and let $v = |V (T)|$ where ($V(A), E(A)$) is the set of (vertices, edges) used by a subgraph $A$. The upper bound on $t$ implies that no $T_j$ can “wrap around,” and so $|E(T_j)| \leq k|V (T_j)|-(2k-1)$ for each $j$ and so 
\beq{3}{
t \leq kv - (2k-1)c.
}
We designate a root vertex $v_j$ for each $T_j$ and order $V(T_j)$ by some order $\prec_j$ that begins with $v_j$ and in which each $v\neq v_j$ appears later than at least one of its neighbors. We may then bound $|\cH\cap\upp{T}|$ as follows. To specify an $S\in\cH$ containing $T$, we first specify a cyclic permutation of $\set{v_1 ,\ldots,v_c} \bigcup ([n]\setminus V (T))$. By \eqref{3}, the number of ways to do this (namely, $(n - v + c - 1)!)$ is at most  $\brac{n-\rdup{\frac{t+(2k-1)c}{k}}+c-1}!$. We then extend to a full cyclic ordering of $[n]$ (thus determining $T$) by inserting, for $j = 1,\ldots,c$, the vertices of $V (T_j ) \setminus \set{v_j}$ in the order $\prec_j$. This allows at most $2k$ places to insert each vertex (since one of its neighbours has been inserted before it and the edge joining them must belong to $T$), so the number of possibilities here is at most $(2k)^v \leq (2k)^{2t}$, and the proposition follows.
\end{proof}
\begin{Proposition}\label{prop2}
For $T\subseteq S\in\cH,|T|=t\leq n/3k$, the number of subgraphs of $T$ with $c$ components is at most $(4ke)^t\binom{2t}{c}$.
\end{Proposition}
\begin{proof}
To specify a subgraph $T$ of $S$ we proceed as follows. We first choose root vertices $v_1 ,\ldots,v_c$ for the components, say $T_1,\ldots,T_c$, of $T$, the number of possibilities for this being at most $\binom{2t}{c}$.
We then choose the sizes, say $t_1,\ldots,t_c$, of $T_1,\ldots,T_c$; here the number of possibilities is at most $\sum_{u=c}^t\binom{u-1}{c-1}$. {(For $u<t$, the summand is the number of positive integer solutions to $x_1+\cdots+x_c=u$.)} Finally, we specify for each $i$ a connected $S_i$ of size $t_i$ rooted at $v_i$ in at most $\prod_{i=1}^c(2ke)^{t_i}$ ways. This comes for the fact that there are at most $(\D e)^{t-1}$ rooted subtrees of the infinite $\D$-regular tree, see Knuth \cite{Kn}, p396, Ex11. Combining these estimates (with  $\sum_{u=c}^t\binom{u-1}{c-1}=\binom{t}{c}<2^t$) yields the proposition.
\end{proof}
It follows from these two propositions that if $S\in\cH$ and $1\leq t\leq n/3k$ then
\begin{align*}
{\frac{f_{t,S}}{|\cH|}}&\leq \sum_{c=1}^{t}(2k)^{2t}\brac{n-\rdup{\frac{t+(2k-1)c}{k}}+c-1}!\times (4ke)^t\binom{2t}{c}\times \frac{1}{(n-1)!}\\
&\leq 2\sum_{c=1}^{t}(16k^3e)^t \binom{2t}{c}\bfrac{e}{n-1}^{\rdup{\frac{t+(2k-1)c}{k}}-c} \bfrac{n-\rdup{\frac{t+(2k-1)c}{k}}+c-1}{n-1}^{n-\rdup{\frac{t+(2k-1)c}{k}}+c-1}\\
&\leq e^{O(t)}  \sum_{c=1}^{t}n^{-\rdup{\frac{t+(2k-1)c}{k}}+c}\\
&=O\bfrac{O(1)}{n^{1/k}}^t.
\end{align*}
So the $k$th power of a Hamiltonian cycle satisfies the conditions with $r=kn$, $\kappa=O(n^{1/k})$, $\a=1/3k$.

\section{Final thoughts}
Theorem \ref{power} could possibly be improved in at least two ways. First, we could try to replace $\e$ by $o(1)$. For specific examples such as the square of a Hamilton cycle, this can probably be done using the ideas of Friedgut \cite{F05}, as suggested in \cite{KNP2020}. Also, we can try to replace $\e_1$ by zero, which would require an improvement to the proof in Section \ref{fin} that we do not have at the moment.

\end{document}